\documentclass[oneside,english,10pt]{amsart}
\usepackage{color}
\usepackage{amsthm}
\usepackage{amstext}
\usepackage{graphicx}
\usepackage{amscd}
\usepackage{amsfonts}
\usepackage[colorlinks,linkcolor=blue, pdfstartview=FitH]{hyperref}
\usepackage{amssymb}
\usepackage{esint}
\usepackage{hyperref}
\usepackage{amsmath, amssymb, amsthm, mathrsfs, graphicx}
\usepackage{hyperref, enumitem, xspace, ifthen,comment}
\usepackage[all]{xy}
\usepackage{tikz}
\usetikzlibrary{arrows,shapes,trees}

\hypersetup{citecolor=red}

\begin{document}

\title[\null]
{Toric Fano manifolds with nef tangent bundles}

\author[\null]{Qilin Yang}

\address{ Department of Mathematics,
Sun Yat-Sen University, ~510275, ~ Guangzhou, ~P. R.  CHINA.} %
\email{yqil@mail.sysu.edu.cn}


\begin{abstract}
In this note we prove that any  toric Fano manifold with nef tangent bundle is a product of projective spaces. In particular, it implies that Campana-Peternell conjecture hold for toric manifolds.
\end{abstract}

\maketitle
\baselineskip=10pt


\newcommand{\C}{\Bbb C}
\newcommand{\kv}{\mathscr{V}}
\newcommand{\kn}{\mathscr{N}}
\newcommand{\ko}{\mathscr{O}}
\newcommand{\kr}{\mathscr{R}}
\newcommand{\kj}{\mathscr{J}}
\newcommand{\kf}{\mathscr{F}}
\newcommand{\Pj}{\Bbb P}
\newcommand{\Z}{\Bbb Z}
\newcommand{\R}{\Bbb R}
\newcommand{\Q}{\ensuremath{\mathbb Q}}
\newcommand{\lnorm}{\left|\left|}
\newcommand{\rnorm}{\right|\right|}
\newcommand{\fk}{\mathfrak{k}}
\newcommand{\fp}{\mathfrak{p}}
\newcommand{\fm}{\mathfrak{m}}
\newcommand{\cf}{\mathrm{cf.}}
\newcommand\grad{\rm grad}
\newcommand\lgr{\longrightarrow}
\newcommand\rw{\rightarrow}
\newcommand\lmp{\longmapsto}
\newcommand\lie{{\cal L}{\rm ie}}
\newcommand\Lm{\Lambda}
\newcommand\lmd{\lambda}
\newcommand\lc{\lceil}
\newcommand\rc{\rceil}
\newcommand\al{\alpha}
\newcommand\bt{\beta}
\newcommand\el{\epsilon}
\newcommand\gm{\gamma}
\newcommand\dt{\delta}
\newcommand\G{\Gamma}
\newcommand\eg{e.g.}
\newcommand\ie{i.e.}
\newcommand\resp{\rm resp.}
\newcommand\Her{\rm Her}
\newcommand\Aut{\rm Aut}
\newcommand\ov{\overline}
\newcommand\be{\beta}
\newcommand\om{\omega}
\newcommand\df{\rm d}
\newcommand\bi{\bar{i}}
\newcommand\bj{\bar{j}}
\newcommand\bl{\bar{l}}
\newcommand\bk{\bar{k}}
\newcommand\bs{\bar{\xi}}
\newcommand\bp{\bar{p}}
\newcommand\gd{\rm grad}
\newcommand\lge{\langle}
\newcommand\rg{\rangle}
\newcommand\sta{\Theta_H }
\newcommand\vk{{\varkappa}_{\tau}}
\newcommand\wh{\widetilde{H}}
\newcommand\wf{\widetilde{F}}

\newcommand\End{\mbox{{\rm End}}\, }
\newcommand\Hom{\mbox{{\rm Hom}}\, }
\newcommand\Ima{\mbox{{\rm Im}}\, }
\newcommand\Ker{\mbox{{\rm Ker}}\, }

\newtheorem{theorem}{Theorem}[section]
\newtheorem{proposition}[theorem]{Proposition}
\newtheorem{corollary}[theorem]{Corollary}
\newtheorem{lemma}[theorem]{Lemma}
\newtheorem{example}[theorem]{Example}
\newtheorem{remark}[theorem]{Remark}
\newtheorem{definition}[theorem]{Definition}

\section{Notation and main result}
We will use standard notation for polytopes and toric varieties, as it can be found in \cite{cls},\cite{fu},\cite{od}.

Let $N\cong\Z^d$ be a $d$-dimensional lattice and $M=\Hom_{\Z}(N,\Z)\cong\Z^d$ the dual lattice with $\langle,\rangle$ the nondegenerate pairing. As usual,
$N_{\Q}=N\otimes_{\Z}\Q\cong\Q^d$ and $M_{\Q}=M\otimes_{\Z}\Q\cong\Q^d$ (respectively $N_{\R}$ and $M_{\R}$) will denote the rational (respectively real) scalar extensions.

A subset $P\subseteq M_{\R}$ is called a polytope if it is the convex hull of finitely many points in $M_{\R}.$
The face of $P$ is denoted by $F\preceq P.$ The set of vertices and facets of $P$ are denoted by $\kv(P)$ and $\kf(P)$ respectively. If $\kv(P)\subseteq M_{\Q}$ (respectively $\kv(P)\subseteq M$) then $P$ is called a rational  polytope (respectively a lattice polytope).

If $P$ is a rational polytope with $0\in {\rm int} P,$ the dual polytope of $P$ is defined by
$$P^*:=\{y\in N_{\R}|\langle x,y\rangle\geq -1,\forall x\in P\},$$
which is also a rational polytope with $0\in {\rm int} P^*.$ The fan $\kn_P:=\{{\rm pos} (F):F\preceq P^*\}$ is called the normal fan of $P.$ Here ${\rm pos}(F)$ denotes the cone positively generated by the face $F$ (also  called positive hull of $F$).  It is well-known that a fan $\Sigma$ in $N_{\R}$ defines a toric variety $X_{\Sigma}:=X(N,\Sigma),$
which automatically admits a torus action and has a Zariski open and  dense orbit:
   $$T_N\times X_{\Sigma}\rw X_{\Sigma},$$
   where $T_N\cong \Hom_{\Z}(M,\C^*).$
   We denote
 $X_P:=X_{\kn_P}$
 the toric variety associated with the normal fan $\kn_P$ of the polytope $P.$
 It is known that $X_P$ is nonsingular if and only if the vertices of any facet of $P^*$ form a $\Z$-basis of the lattice $M.$

A $d$-dimensional polytope $P\subseteq M_{\R}$ with $0\in {\rm int} P$  is called reflexive polytope if both $P$ and $P^*$ are lattice polytopes. A complex variety $X$  is called a Gorenstein Fano variety if $X$ is projective, normal and its anticanonical divisor is an ample Cartier divisor. The following theorem
(see \cite{ni1}) classifies  Gorenstein toric Fano varieties by reflexive polytopes:

\begin{theorem}\label{refx}Under the map $P\longmapsto X_P$ reflexive polytopes correspond uniquely up to isomorphism to Gorenstein toric Fano varieties. There are only finitely many isomorphism types of $d$-dimensional reflexive polytopes.
\end{theorem}

A Cartier divisor $D$ on a nonsingular variety $X$ is called a nef divisor if the intersectional number $D\cdot C\geq 0$ for any irreducible curve $C\subset X.$
A line bundle $L$ is called a nef line bundle if the associated Cartier divisor (i.e., $L=\ko_X(D)$) is a nef divisor. A vector bundle $E$ over $X$ is called a nef vector bundle if the tautological line bundle $\ko_{\Pj(E^*)}(1)$ on the projective bundle $P(E^*)$ is a nef line bundle.  In \cite{cp}, Campana and Peternell conjectured that any  Fano manifold with nef tangent bundle is a rational homogeneous manifold. In this note we confirm this conjecture for toric Fano manifold, in fact we get more and obtain the following main theorem:

\begin{theorem}\label{main} Any  toric Fano manifold with nef tangent bundle is a product of projective spaces.
\end{theorem}

\section{Cartier divisors on complete toric varieties}
A fan $\Sigma$ in $N_{\R}$ is complete iff its support $|\Sigma|:=\cup_{\sigma\in\Sigma}\sigma$ is the whole space $N_{\R},$ which is also equivalent to that
the associated toric variety $X(N,\Sigma)$ is compact in classical topology {\rm (\cite[Theorem 3.1.9]{cls})}.

 Let $\Sigma(k)$ denote the set of $k$-dimensional cones of the complete fan $\Sigma.$ The elements in $\Sigma(1)$ are called rays, and given $\tau\in \Sigma(1),$ let $u_{\tau}$ denote the unique minimal generator of $N\cap \tau.$ By orbit-cone correspondence {\rm (\cite[Theorem 3.2.6]{cls})}, a ray $\tau\in \Sigma(1)$ gives a $T_N$-invariant Cartier divisor $D_{\tau}.$ On a complete toric variety
we may write any Cartier divisor as a linear combination of $T_N$-invariant Cartier divisors. Let $D=\sum_{\tau\in \Sigma(1)}a_{\tau}D_{\tau}$ be a Cartier divisor on a complete toric variety $X_{\Sigma},$ its support function $\phi_D:N_{\R}\rw \R$ is determined by the following properties:
\begin{enumerate}
\item $\phi_D$ is linear on each cone $\sigma\in\Sigma.$
\item $\phi_D(u_{\tau})=-a_{\tau}.$
\item For each cone $\sigma\in\Sigma$ there is a $m_{\sigma}\in M$ such that $\phi_D (u)=\langle m_{\sigma},u\rangle$ for all $u\in \sigma$
and $\langle m_{\sigma},u_{\tau}\rangle=-a_{\tau}$ for all $\tau\in\sigma(1).$
\end{enumerate}
\begin{proposition}\label{bas}{\rm (\cite[Theorem 6.1.10 and Theorem 6.2.12]{cls})} Let $D=\sum_{\tau\in \Sigma(1)}a_{\tau}D_{\tau}$ be a  Cartier divisor  on a complete toric variety $X_{\Sigma}$ and denote $$P_D=\{m\in M_{\R}|\langle m,u_{\tau}\rangle\geq -a_{\tau},\forall \tau\in \Sigma(1)\}.$$
 Then the following are equivalent:
\begin{enumerate}
\item $D$ is basepoint free.
\item $D$ is a nef divisor.
\item $\phi_D$ is a upper convex function.
\item \label{ddd} $m_{\sigma}\in P_D$ for all $\sigma\in\Sigma(d).$
\item $\phi_D(u)={\rm min}_{m\in P_D}\langle m,u\rangle$ for all $u\in N_{\R}.$
\end{enumerate}
 \end{proposition}
 The support function
$\phi_D$ of a Cartier divisor $D$ on a complete toric variety $X_{\Sigma}$ is called {\it strictly convex} if it is upper convex and for each $\sigma\in \Sigma(d)$ satisfies
$$\langle m_{\sigma},u\rangle=\phi_D(u)\Longleftrightarrow u\in\sigma.$$
\begin{proposition}\label{str} {\rm (\cite[Theorem 6.1.15 and Corollary 6.1.16]{cls})}A Cartier divisor $D$ on a complete toric variety $X_{\Sigma}$ is ample if and only if its support function $\phi_D$ is strictly convex. If $D$ is ample then $P_D$ is a full dimensional lattice polytope whose normal fan is $\Sigma.$
\end{proposition}

\section{Complete toric variety with reductive automorphism group}
The automorphism group $\Aut (X_{\Sigma})$ of  a nonsingular complete toric variety $X_{\Sigma}$ was firstly studied by Demazure in \cite[Section 4]{de}.
Identifying the elements of the Lie algebra of  $\Aut (X_{\Sigma})$ with the invariant differential operators on the coordinate ring of $X_{\Sigma},$
Demazure gave a very simple description of the structure of the Lie algebra of $\Aut (X_{\Sigma})$ using the ¡¤¡¤Demazure root system¡° named after him.
The Demazure root system $\kr$  of $\Aut (X_{\Sigma})$ has a very simple description:
$$\kr=\{m\in M|\exists \tau\in \Sigma(1) : \langle u_{\tau},m\rangle=-1,\langle u_{\tau'},m\rangle\geq 0,\forall\tau'\in\Sigma(1)\backslash\{\tau\}\}.$$
Note here we use notation of \cite{ni1,ni2}, which are different form those in \cite{de} by a ¡±minus¡° signature.
The Demazure roots in $\kr\cap-\kr=\{m\in\kr|-m\in\kr\}$ are called semisimple roots.
 $\Aut (X_{\Sigma})$ is a reductive algebraic group iff all Demazure roots in $\kr$ are semisimple, i.e., $\kr=-\kr:=\{-m|m\in\kr\}.$ The following proposition of Nill, Benjamin's will be used in proving our main theorem in the next section.
 \begin{proposition}\cite[Proposition 3.18]{ni2}\label{red} A $d$-dimensional complete toric variety is isomorphic to a product of projective spaces iff there are $d$-linearly independent semisimple roots.
 \end{proposition}

 \section{Toric Fano manifolds with nef tangent bundles}
 The projective space $\Pj^n$ is a toric Fano manifold and the following exact sequence
$$\xymatrix{
0\ar[r]&\ko_{\Pj^d}\ar[r]& \ko_{\Pj^n}(1)^{d+1}\ar[r] & T_{\Pj^d}\ar[r] & 0
}\label{pj}$$
is  called the Euler sequence of $\Pj^d.$ The following theorem is a toric generalization of this result.

\begin{theorem}\label{tan}{\rm (\cite[Theorem 8.1.6]{cls})} Let $X_{\Sigma}$ be a  toric  manifold associated with the complete fan $\Sigma,$ then we have the following generalized Euler sequence
$$\xymatrix{
0\ar[r]&\ko^{\oplus \rho}_{X_{\Sigma}}\ar[r]& \oplus_{\tau\in\Sigma(1)}\ko_{X_{\Sigma}}(D_{\tau})\ar[r] & T_{X_{\Sigma}}\ar[r] & 0,
}$$
where $\rho$ is the Picard number of $X_{\Sigma}.$
\end{theorem}
In particular, the canonical divisor and anticanonical divisor of the  toric  manifold $X_{\Sigma}$ are respectively given by
 $$K=-\sum _{\tau\in\Sigma(1)}D_{\tau};\quad\quad K^*=\sum _{\tau\in\Sigma(1)}D_{\tau}.$$
 \begin{proposition}
 Assume $X_{\Sigma}$ is a  toric manifold with nef tangent bundle, then for any $\tau\in \Sigma(1),$ the associated $T_N$-invariant Cartier divisor $D_{\tau}$ is a nef divisor.
\end{proposition}
\begin{proof} Since $D_{\tau}$ is $T_N$-invariant, it is a smooth hypersurface locating inside  $X_{\Sigma},$  we have the following exact sequnce
$$\xymatrix{
0\ar[r]& T_{D_{\tau}}\ar[r]& T_{X_{\Sigma}}\Big|_{D_{\tau}}\ar[r] & N_{D_{\tau}}\ar[r] & 0,
}$$
Note the normal sheaf  of $D_{\tau}$ in $X_{\Sigma}$ could be identified with $\ko_{X_{\Sigma}}(D_{\tau}).$ Since the tangent bundle $T_{X_{\Sigma}}$ is a nef vector bundle, so is the quotient bundle $N_{D_{\tau}}$ by \cite[Proposition 1.15]{dps}. Hence $D_{\tau}$ is a nef divisor.
\end{proof}
\begin{theorem}\label{gri} The tangent bundle of a  toric manifold with nef tangent bundle is Griffiths semipositive.
\end{theorem}
\begin{proof} Since $\ko_{X_{\Sigma}}(D_{\tau})$ is a nef line bundle, $D_{\tau}$ is basepoint free by Proposition \ref{bas}. Hence $\ko_{X_{\Sigma}}(D_{\tau})$
is a semipositive line bundle. Hence the direct sum bundle $\oplus_{\tau\in\Sigma(1)}\ko_{X_{\Sigma}}(D_{\tau})$ is a Griffiths semipositve.
The Griffiths semipositivity of  tangent bundle $ T_{X_{\Sigma}}$ follows by the generalized Euler sequence via using \cite[Proposition 3.5]{ya}.
\end{proof}
Note Theorem \ref{gri} already implies that Campana-Peternell conjecture holds for  toric fano manifolds. In fact, from Theorem \ref{gri} we know  a toric fano manifold $X_{\Sigma}$ with nef tangent bundle has nonnegative holomorphic bisectional curvature and positive Ricci curvature. By Mok's theorem \cite{mo}, $X_{\Sigma}$ is biholomorphic to the product of Hermitian symmetric manifolds. In particular, $\Aut (X_{\Sigma})$ is a reductive algebraic group. In the rest part of this note we will give a more precise structure description of a  toric fano manifold with nef tangent bundle.

Let $\phi_{D_{\tau}}$ be the support function associated with $T_N$-invariant cartier divisor $D_{\tau}.$ Then on each open cone $\sigma\in\Sigma(d),$
$$\phi_{D_{\tau}}(u)=\langle m_{\sigma},u \rangle,~~~\forall u\in \sigma,$$
for some $m_{\sigma}\in M.$

\begin{proposition}\label{sm} If the cartier divisor $D_{\tau}$ is nef then $\{m_{\sigma}|\sigma\in\Sigma(d)\}$ are semisimple Demazure roots of $\Aut (X_{\Sigma}).$
\end{proposition}
\begin{proof}  By the definition of the data $m_{\sigma}$ for the Cartier divisor $D_{\tau}$, we have $\phi_{D_{\tau}}(u_{\tau})=\langle m_{\sigma},u_{\tau}\rangle=-1.$  Now fix a $\sigma\in \Sigma(d).$ By (\ref{ddd}) of Proposition \ref{bas},    $m_{\sigma}\in P_{D_{\tau}}$ if $D_{\tau}$ is basepoint free. Note now
$$ P_{D_{\tau}}=\{m\in M|\langle m,u_{\tau}\rangle\geq -1 ~~{\rm and}~~\langle m,u_{\tau'}\rangle\geq 0,\forall \tau'\in \Sigma(1)\backslash\{\tau\}\},$$
 hence we have $$\langle m_{\sigma},u_{\tau'}\rangle\geq 0,\forall \tau'\in \Sigma(1)\backslash\{\tau\}\},$$
therefore $m_{\sigma}$ is a Demazure root of $\Aut (X_{\Sigma}).$ Since $\Aut (X_{\Sigma})$ is reductive, it is also a semisimple root.
\end{proof}

\begin{proposition}\label{ind} If $X_{\Sigma}$ is a toric Fano manifold with nef tangent bundle then  $\Aut (X_{\Sigma})$ has $d$ linearly independent semisimple roots.
\end{proposition}
\begin{proof} Let $\tau_1,\cdots,\tau_m$ denote all  of $1$-dimensinal cones of $\Sigma$ and $u_{\tau_1},\cdots,u_{\tau_m}$ their primitive generating vectors,
 and $D_1,\cdots,D_m$ the corresponding $T_N$-invariant basepoint free Cartier divisors. The support function $\phi_{D_i}$ of $D_i$ satisfies that
$$\phi_{D_i}(x)=\langle m_{\sigma},x\rangle,\quad \forall x\in \sigma\in\Sigma(d),$$
where $\langle m_{\sigma},u_{\tau_i}\rangle=-1$ and $\langle m_{\sigma},u_{\tau_j}\rangle\geq 0$ for $j\not=i.$ Let $\{m^i_{\sigma}\}$ be the set of semisimple Demazure roots associated with the Cartier divisor $D_i.$
   Since
each $D_i$ is a basepoint free divisor, the associated polytope
$$P_{D_i}=\{m\in M_{\R}|\langle m,u_{\tau_i}\rangle\geq -1~~{\rm and}~~ \langle m,u_{\tau_j}\rangle\geq 0~~\forall~~j\not=i\}$$
is a convex polytope. Note for $i\not=j,$
$$P_{D_i}\cap P_{D_j}=\{m\in M_{\R}| \langle m,u_{\tau_j}\rangle\geq 0~~{\rm for}~~j=1,\cdots,m\}={\rm pos}(\tau_1,\cdots,\tau_m)^{\vee}$$
is $\{0\}$, since $\Sigma$ is complete the convex cone ${\rm pos}(\tau_1,\cdots,\tau_m)=N_{\R}.$

The anticanonical divisor of $X_{\Sigma}$ is given by $K^*=D_1+\cdots+D_m$ and it is an ample divisor, the associated polytope
$$P_{K^*}=\{m\in M_{\R}|\langle m,u_{\tau_i}\rangle\geq -1~~{\rm for}~~i=1,\cdots,m\}$$
is a full  dimensional polytope by Proposition \ref{str}. Note that
$$P_{K^*}=P_{D_1}\cup\cdots\cup P_{D_m}.$$
Since $0\not=m^i_{\sigma}\in P_{D_i},$ we have for any $\sigma,\sigma'\in\Sigma(d)$ that $m^i_{\sigma}\not=m^j_{\sigma'}$ if $i\not=j.$ Note $\{m^i_{\sigma}\}$ are vertices  of $P_{D_i},$ however none of them are vertices of $P_{K^*}$ though $\{m^i_{\sigma}|\sigma\in\Sigma(d)\} \subset P_{K^*}.$ In fact  $m^i_{\sigma}$ can't  lie in the intersection of two facets of $P_{K^*},$ hence it is not inside the facets  with codimension $\geq 2$ of $P_{K^*}.$
 But each $m^i_{\sigma}$ is in the codimensional one facet $H_i=\{x\in M|\langle m,u_{\tau_i}\rangle= -1\}$ of $P_{K^*},$ and for $i\not=j,$ the points $\{m^i_{\sigma}|\sigma\in\Sigma(d)\}$ and $\{m^j_{\sigma}|\sigma\in\Sigma(d)\}$ locate in the different facets of $P_{K^*}.$

 Now let $v$ be any vertex of $P_{K^*}$ which has at least $d$ codimensional one facets of $P_{K^*},$ assume $H_{i_1},\cdots,H_{i_k}(k\geq d)$ are those facets passing through the vertex $v$ and $H_{i_1}\cap\cdots\cap H_{i_k}=\{v\}.$ Now fix a cone $\sigma\in \Sigma(d),$ since $\{m\in M|\langle m,u_{\tau_{i_j}}\rangle\geq  -1,j=1,\cdots,k\}$ is a $d$-dimensional cone ${\rm Cone}(P_{K^*}\cap M_{\R}-v),$ the vectors
 $m^{i_1}_{\sigma}-v,\cdots,m^{i_k}_{\sigma}-v$  form a basis of $N_{\R}.$ Since $m^{i_1}_{\sigma},\cdots,m^{i_k}_{\sigma}$ are on the different facets of cone
 ${\rm Cone}(P_{K^*}\cap M_{\R}-v),$
 without loss of generality we may assume $m^{i_1}_{\sigma}-v,\cdots,m^{i_d}_{\sigma}-v$are linearly independent. Then after a translation, $m^{i_1}_{\sigma},\cdots,m^{i_d}_{\sigma}$ are still linearly independent. By Proposition \ref{sm},$m^{i_1}_{\sigma},\cdots,m^{i_d}_{\sigma}$ are semisimple Demazure roots of
 ${\rm Aut}(X_{\Sigma}),$ hence it has $d$ linearly independent semisimple roots.
\end{proof}

Now  our main result Theorem \ref{main} follows from Proposition \ref{red} and Proposition \ref{ind}.

 \end{document}